\documentclass[12pt]{article}
\usepackage[text={6in,8.5in},centering]{geometry}
\font\smallit=cmti10

\usepackage{amssymb,latexsym,amsmath,epsfig,amsthm} 
\usepackage{hyperref, enumerate}

\makeatletter

\renewcommand\section{\@startsection {section}{1}{\z@}
{-30pt \@plus -1ex \@minus -.2ex}
{2.3ex \@plus.2ex}
{\normalfont\normalsize\bfseries\boldmath}}

\renewcommand\subsection{\@startsection{subsection}{2}{\z@}
{-3.25ex\@plus -1ex \@minus -.2ex}
{1.5ex \@plus .2ex}
{\normalfont\normalsize\bfseries\boldmath}}

\renewcommand{\@seccntformat}[1]{\csname the#1\endcsname. }

\makeatother

\newtheorem{theorem}{Theorem}
\newtheorem{lemma}[theorem]{Lemma}
\newtheorem{proposition}[theorem]{Proposition}
\newtheorem*{theorem*}{Theorem}

\newcommand{\seqnum}[1]{\href{http://oeis.org/#1}{{#1}}}

\newcommand{\sumz}[1]{\sum_{#1=0}^\infty}

\begin{document}

\begin{center}
\uppercase{\bf On the Almkvist--Meurman theorem for Bernoulli polynomials}
\vskip 20pt
{\bf Ira M. Gessel%
  \footnote{Supported by a grant from the Simons Foundation (\#427060, Ira Gessel)}}\\
{\smallit Department of Mathematics, Brandeis University, Waltham, MA, USA}\\
{\tt gessel@brandeis.edu}\\ 
\end{center}
\vskip 10pt
%\centerline{\smallit Received: , Revised: , Accepted: , Published: } 
% We will fill in the dates
%\vskip 30pt

\centerline{\bf Abstract}
\noindent
Almkvist and Meurman showed that if $h$ and $k$ are integers, then so is \[k^n\bigl(B_n(h/k) - B_n\bigr),\] where $B_n(u)$ is the Bernoulli polynomial. We give here a new and simpler proof of the Almkvist--Meurman theorem using generating functions. We describe some properties of these numbers and prove a common generalization of the Almkvist--Meurman theorem and a result of Gy on Bernoulli--Stirling numbers. We then give a simple generating function proof of an analogue of the Almkvist--Meurman theorem for Euler polynomials, due to Fox.
\pagestyle{myheadings}
%\markright{\smalltt INTEGERS: 23 (2023)\hfill}
\thispagestyle{empty}
\baselineskip=12.875pt
\vskip 30pt

\section{The Almkvist--Meurman Theorem}
\subsection*{Introduction}
Let $B_n(u)$ denote the $n$th Bernoulli polynomial, defined by  
\begin{equation}
\label{e-bp}
\sum_{n=0}^\infty B_n(u) \frac{x^n}{n!}=\frac{xe^{ux}}{e^x-1}.
\end{equation}
Then $B_n(0)$ is the $n$th Bernoulli number $B_n$, defined by 
\begin{equation*}
\sum_{n=0}^\infty B_n \frac{x^n}{n!}=\frac{x}{e^x-1}.
\end{equation*}
These generating functions may be viewed as formal power series or as convergent series for $|x| < 2\pi$.

For $k\ne0$, let
\begin{equation}
\label{e-Mdef}
M_n(h,k)=k^n\bigl(B_n(h/k) - B_n\bigr)=\sum_{i=0}^{n-1}\binom{n}{i}B_i h^i k^{n-i}.
\end{equation}
Almkvist and Meurman  \cite{am} showed in 1991 that if $h$ and $k$ are integers then $M_n(h,k)$ is an integer. Other proofs were given by Sury \cite{sury},
 Bartz and Rutkowski \cite{br}, and Clarke and Slavutskii \cite{cs}.
 
 We give here a  simple new  proof, using generating functions, of the Almkvist--Meurman theorem, and we discuss some of the properties of these integers. 
 We also prove a common generalization of the Almkvist--Meurman theorem and a theorem of Gy \cite{gy20} on ``Bernoulli-Stirling numbers," and we give a simple generating function proof of an analogue of the Almkvist--Meurman theorem  for Euler polynomials due to Fox \cite{fox}.
  
\subsection*{Vandiver's Theorem}
A closely related result, from which Almkvist and Meurman's result follows easily, was proved by Vandiver much earlier. (Previous authors on the Almkvist--Meurman theorem seem to  be unaware of Vandiver's work.) Vandiver stated his result and gave a brief indication of the proof  in 1937 \cite{vandiver1937} and gave a complete proof using a different approach in 1941 \cite[Theorem III]{vandiver1941}. Carlitz \cite[Theorem 2]{carlitz} gave another proof of Vandiver's theorem. We state Vandiver's theorem here and explain how the Almkvist--Meurman theorem follows from it.
\begin{theorem*}[Vandiver]
Let $h$ and $k$ be integers with $k\ne0$.
 If $n$ is even  and positive then
\begin{equation}
\label{e-vand}
k^n B_n(h/k) = G_n - \sum_p \frac{1}{p}
\end{equation}
where $G_n$ is an integer and the sum is over all primes $p$ such that $p-1\mid n$ but $p\nmid k$. If $n$ is odd then $k^n B_n(h/k)$ is an integer unless $n=1$ and $k$ is odd, and in this case $kB_1(h/k) =G_1+1/2$, where $G_1$ is an integer. 
\end{theorem*}
Since $B_n(0) = B_n$, the case $h=0, k=1$ of Vandiver's theorem is the well-known von Staudt--Clausen theorem \cite{carlitz}, which describes the fractional part of the Bernoulli numbers. 
Vandiver's theorem implies that, in all cases, the fractional part of $k^nB_n(h/k)$ is independent of $h$ and thus is equal to the fractional part of $k^nB_n(0) = k^n B_n$. Therefore $k^n (B_n(h/k) - B_n)$ is an integer. 
Conversely, Vandiver's theorem may be derived easily from the Almkvist--Meurman theorem together with the von Staudt--Clausen theorem. The Almkvist--Meurman theorem implies that the fractional part of $k^nB_n(h/k)$ is the same as the fractional part of $k^nB_n$, and the fractional part of $k^nB_n$ is easily determined by the von Staudt-Clausen theorem. 
Vandiver's theorem was rediscovered by Bartz and Rutkowski  \cite{br}, who derived the Almkvist--Meurman theorem from it.

\subsection*{A generating function proof of the Almkvist--Meurman theorem}
\label{s-AMproof}

A \emph{Hurwitz} series is a power series $\sum_{n=0}^\infty r_n x^n/n!$ for which each $r_n$ is an integer; i.e., it is the exponential generating function for a sequence of integers.  It is well known that Hurwitz series are closed under addition and multiplication, and that if $f(x)$ is a Hurwitz series with constant term 0 then $f(x)^k/k!$ is a Hurwitz series. In particular, the Stirling numbers of the second kind $S(n,k)$, defined by 
\begin{equation*}
\frac{(e^x-1)^k}{k!}=\sum_{n=k}^\infty S(n,k) \frac{x^n}{n!},
\end{equation*}
are integers.

Since all of the generating function we will be concerned with are exponential, by the ``coefficients" of a generating function we mean its coefficients as an exponential generating function. 

It follows from Equation \eqref{e-bp} and Equation \eqref{e-Mdef} that
\begin{equation}
\label{e-AMgf}
\sum_{n=0}^\infty M_n(h,k) \frac{x^n}{n!} 
 =\sum_{n=0}^\infty  k^n\bigl(B_n(h/k) - B_n\bigr) \frac{x^n}{n!}
 =  kx \frac{e^{hx}-1}{e^{kx}-1}.
 \end{equation}
We present here a new proof of the Almkvist--Meurman theorem using the generating function of Equation \eqref{e-AMgf} and basic facts about Hurwitz series. 
\begin{theorem}[Almkvist--Meurman] 
\label{t-am}
For all integers  $h$ and  $k$, with $k\ne0$,  $M_n(h,k)$ is an integer.
\end{theorem}
\begin{proof} 
From Equation \eqref{e-AMgf} we have
\begin{equation*}
\sum_{n=0}^\infty M_n(h,k) \frac{x^n}{n!} 
 =  kx \frac{e^{hx}-1}{e^{kx}-1}
  =kx\frac{e^x-1}{e^{kx}-1}\cdot\frac{e^{hx}-1}{e^x-1}.
\end{equation*}
If $h\ge0$ then $(e^{hx}-1)/(e^x-1)$ is a Hurwitz series, since
\begin{equation*}
\frac{e^{hx}-1}{e^x-1}=1+e^x+\cdots+e^{(h-1)}.
\end{equation*}
If $h<0$ then 
\begin{equation*}
\frac{e^{hx}-1}{e^x-1}=-e^{hx}\frac{e^{-hx}-1}{e^x-1},
\end{equation*}
so $(e^{hx}-1)/(e^x-1)$ is also a Hurwitz series in this case.
Thus it is sufficient to show that $kx(e^x-1)/(e^{kx}-1)$ is a Hurwitz series. 

We have
\begin{align}
kx \frac{e^{x}-1}{e^{kx}-1} &= 
\frac{e^{x}-1}{e^{kx}-1}\log \bigl(1+ (e^{kx} -1)\bigr)\notag\\
&=\frac{e^{x}-1}{e^{kx}-1}
  \sum_{j=1}^\infty (-1)^{j-1}\frac{(e^{kx}-1)^j}{j}\notag\\
&=\sum_{j=1}^\infty (-1)^{j-1}\frac{(e^x-1)^j}{j} \left(\frac{e^{kx}-1}{e^x-1}\right)^{j-1}\notag\\
&=\sum_{j=1}^\infty (-1)^{j-1}(j-1)!\,\frac{(e^x-1)^j}{j!} \left(\frac{e^{kx}-1}{e^x-1}\right)^{j-1}.\label{e-a1}
\end{align}
Since $(e^x-1)^j/j!$ and $(e^{kx}-1)/(e^x-1)$ are Hurwitz series, so is the sum in Equation \eqref{e-a1}.
\end{proof}

\subsection*{The Almkvist--Meurman numbers}

In this section we discuss the Almkvist--Meurman numbers $M_n(h,k)$. 

\begin{proposition}
The numbers $M_n(h,k)$ have the following properties:
\begin{enumerate}
\item[\textup{(a)}]
For $h>0$ we have
\begin{equation*}
M_n(h,1) =n(0^{n-1}+1^{n-1}+\cdots+(h-1)^{n-1}).
\end{equation*}

\item[\textup{(b)}]
As a polynomial in $h$ and $k$, $M_n(h,k)$ is  homogeneous  of degree $n$.

\item[\textup{(c)}]
We have 
\begin{gather}
M_n(k-h,k) = (-1)^n M_n(h,k)\text{ for $n\ne 1$,}\label{e-Mc1}\\
M_n(h+k,k) = M_n(h,k) + nkh^{n-1}\text{ for all $n$.}\label{e-Mc2}
\end{gather}
\end{enumerate}
\end{proposition}

\begin{proof}\leavevmode\vskip-\parskip

(a) By Equation \eqref{e-AMgf} we have $\sumz n M_n(h,1) x^n/n!=x(1+e^x+\cdots+e^{(h-1)x})$.

(b) This follows immediately from Equation \eqref{e-Mdef}.

(c)
By the identity $B_n(1-u) = (-1)^n B_n(u)$ for Bernoulli polynomials, we have
\begin{equation*}
k^n \left(B_n\Bigl(\frac{k-h}{k}\Bigr) -B_n\right)=(-1)^n k^n \left(B_n\Bigl(\frac{h}{k}\Bigr) -(-1)^nB_n\right).
\end{equation*}
Then Equation \eqref{e-Mc1} follows since $B_n=0$ if $n$ is odd and not equal to 1. Equation \eqref{e-Mc2} follows similarly from the identity $B_n(u+1) = B_n(u) +nu^{n-1}$.
\end{proof}

The first few values of $M_n(h,k)$ as  polynomials in $h$ and $k$ are
\begin{align*}
M_0(h,k) &=0,\\
M_1(h,k) &= h,\\
M_2(h,k)&=h^2-hk = -h(k-h),\\
M_3(h,k) &=h^3 -\tfrac32 h^2k+\tfrac12 hk^2 =\tfrac12h(k-h)(k-2h),\\
M_4(h,k) &= h^4-2h^3k+h^2k^2=h^2(k-h)^2,\\
M_5(h,k) &=h^5 -\tfrac52h^4k+\tfrac53h^3k^2-\tfrac16hk^4=-\tfrac16h(k-h)(k-2h)(k^2+3hk-3h^2),\\
M_6(h,k)&=h^6-3h^5k+\tfrac52h^4k^2-\tfrac12h^2k^4=\tfrac12 h^2(k-h)^2(k^2+2hk-2h^2).
\end{align*}

\begin{proposition} As a polynomial in $h$ and $k$, 
$M_n(h,k)$ has the following divisibility properties:
\begin{enumerate}
\item[\textup{(a)}] $M_n(h,k)$ is divisible by $h$ for all $n$.
\item[\textup{(b)}] $M_n(h,k) -h^n$ is divisible by $hk$ for $n\ge1$.
\item[\textup{(c)}] $M_n(h,k)$ is divisible by $k-h$ for $n>1$.
\item[\textup{(d)}] $M_n(h,k)$ is divisible by $k-2h$ for $n$ odd and greater than $1$.
\item[\textup{(e)}] $M_n(h,k)$ is divisible by $h^2(k-h)^2$ for $n$ even and greater than $2$.
\end{enumerate}
\end{proposition}

\begin{proof}
The proofs of (a), (b), (c), and (d) are straightforward. For example, (a) is equivalent to $M_n(0,k)=0$, which is clear either from the generating function or from Equation \eqref{e-Mdef} and $B_n(0) = B_n$.

For (e), it is enough to show that for $n$ even and greater than 2, $B_n(u) -B_n$ is divisible by $u^2(1-u)^2$. Suppose that $n$ is even and greater than 2.  Since $B_n(u)= 
B_n +n B_{n-1}u+\cdots + u^n$ and $B_{n-1}=0$, it follows that $B_n(u) - B_n$ is divisible by $u^2$. Since $B_n(u) = B_n(1-u)$, $B_n(u) -B_n$ is also divisible by $(1-u)^2$.
\end{proof}

We note that by (b), as a polynomial in $h$ and $k$, $M_n(h,0) = h^n$ for $n>0$.
Empirically, it seems that for $n$ even, $(-1)^{n/2}M_n(a,a+b)$ is a polynomial in $a$ and~$b$ with positive coefficients, but the coefficients are not all integers.

Let $A_n(k)=M_n(1,k)$, so
\begin{equation}
\label{e-gf}
\sum_{n=0}^\infty A_n(k) \frac{x^n}{n!}=kx \frac{e^x -1}{e^{kx}-1}=\frac{kx}{1+e^x+\cdots + e^{(k-1)x}}.
\end{equation}
Some values of $A_n(k)$ are given in Table \ref{tab-Table1}.
\begin{table}[h!]  
\[  
\label{tab-Table1}  
\begin{array}{c|rrrrrrrrrr}  
k\backslash n&1&2&3&4&5&6&7&8&9&10\\
\hline
2&1&-1&0&1&0&-3&0&17&0&-155\\
3&1&-2&1&4&-5&-26&49&328&-809&-6710\\
4&1&-3&3&9&-25&-99&427&2193&-12465&-79515\\
5&1&-4&6&16&-74&-264&1946&9056&-88434&-512024\\
6&1&-5&10&25&-170&-575&6370&28225&-415826&-2294975\\
\end{array}
\]
\caption{Values of $A_n(k)$} 
\end{table}
The rows in this table are sequences \seqnum{A083007} through \seqnum{A083014} in the 
On-Line Encyclopedia of Integer Sequences \cite{oeis}. In particular, the numbers $A_n(2)$ are the well-known Genocchi numbers (\seqnum{A036968}; see also \seqnum{A001469} and \seqnum{A110501}) with exponential generating function $2x/(e^x+1)$. The numbers $(-1)^{n+1}A_{2n+1}(3)$ are Glaisher's G-numbers (\seqnum{A002111}).

It is apparent from the table that  $(-1)^{\lceil n/2\rceil}A_n(k)$ is positive  for $n>1$. We will prove this in Theorem \ref{t-pos} below. 
To do this we use properties of Bernoulli polynomials. These properties are known  (see N\"orlund \cite[pp.~22--23]{norlund}), but we include proofs for completeness.
\begin{lemma}
\label{l-1}
\mbox{}
\begin{enumerate}
\item[\textup{(a)}]For $m\ge1$, $(-1)^m B_{2m}(u)$ is increasing on $[0,1/2]$.
\item[\textup{(b)}] For $m\ge1$, $(-1)^m B_{2m-1}(u)$ is positive and $(-1)^{m+1} B_{2m+1}''(u)$ is negative
 on $(0,1/2)$. \end{enumerate}
\end{lemma}

\begin{proof} From the definition of the Bernoulli polynomials in Equation \eqref{e-bp} it follows that 
\begin{equation}
\label{e-diff}
B_n'(u) = nB_{n-1}(u). 
\end{equation}
We will also  use the well-known facts that $B_n=0$ for $n$ odd and greater than 1 and $B_n(1/2)=0$ for all odd $n$.

We proceed by induction on $m$. We have $-B_1(u) = 1/2 -u$, so $-B_1(u)$ is positive on $(0,1/2)$. Now suppose that $m\ge1$ and that $(-1)^{m}B_{2m-1}(u)$ is positive on $(0,1/2)$. It follows from 
Equation \eqref{e-diff} that $(-1)^{m}B_{2m}(u)$ is increasing on $[0,1/2]$. 

Setting $p(u) = (-1)^{m+1}B_{2m+1}(u)$, we have $p(0) = p(1/2) = 0$ and \[p''(u) = (-1)^{m+1}(2m+1)(2m)B_{2m-1}(u),\] which is negative on $(0,1/2)$. So $p(u)$ is strictly concave on $[0,1/2]$, and therefore takes on its minimum values only at the endpoints of this interval. Thus $p(u)$ is positive on $(0,1/2)$.
\end{proof}

\begin{lemma}
\label{l-2}
Let $\tilde{B}_n(u) = B_n(u) - B_n$. 
\begin{enumerate}
\item[\textup{(a)}]
If $n$ is odd and greater than 1 then $(-1)^{\lceil n/2\rceil}\tilde{B}_n(u)$ is positive for $0<u<1/2$. 
\item[\textup{(b)}] If $n$ is even then  $(-1)^{\lceil n/2\rceil}\tilde{B}_n(u)$ is positive for $0<u<1$.
\item[\textup{(c)}] If $n$ is divisible by $4$ then $\tilde{B}_n(u)$ is nonnegative for all $u$.
\end{enumerate}
\end{lemma}

\begin{proof} Part (a) follows immediately from (b) of Lemma \ref{l-1}, since for $n$ odd and greater than~1, $\tilde{B}_n(u) = B(u)$.  Part (b) follows from (a) of Lemma \ref{l-1} together with the facts that $\tilde{B}_n(0) = 0$ and $B_n(1-u) = (-1)^n B_n(u)$. Part (c) follows from part (b) and the formula $B_n(u+1) = B_n(u) +n u^{n-1}$.
\end{proof}

Lemma \ref{l-2} implies the following positivity results for $M_n(h,k)$.

\begin{theorem}
\label{t-pos}
Let $k$ be a positive integer.  
\begin{enumerate}
\item[\textup{(a)}]
If $n$ is odd and greater than $1$, and $0<h<k/2$, then $(-1)^{\lceil n/2\rceil}M_n(h,k)$ is positive.
\item[\textup{(b)}] If $n$ is even and $0<h<k$ then  $(-1)^{\lceil n/2\rceil}M_n(h,k)$ is positive.
\item[\textup{(c)}] If $n$ is divisible by $4$ then $M_n(h,k)$ is nonnegative for all $h$.
\end{enumerate}
\end{theorem}

\section{A generalization of Gy's theorem}
Gy \cite[Theorems 6.1 and 6.2]{gy20} has studied the numbers with exponential generating function 
\begin{equation*}
\frac{(e^x -1)^k}{k!} \frac{kx}{e^{kx}-1}
\end{equation*}
and has shown that they are integers.  Comparison with the Almkvist--Meurman theorem suggests that we look at the more general exponential generating function
\begin{equation}
\label{e-gy1}
\frac{(e^{hx} -1)^j}{j!} \frac{kx}{e^{kx}-1},
\end{equation} 
which reduces to Gy's generating function for $h=1, j=k$ and to the Almkvist--Meurman generating function for $j=1$.
The coefficient of $x^n/n!$ in \eqref{e-gy1} is 
\begin{equation}
\label{e-gyc}
\sum_{i=0}^{n-j} \binom{n}{i}h^{n-i} S(n-i,j) k^{i}B_{i}.
\end{equation}
It is not true that these numbers are always integers. However, in Theorem \ref{t-GAM} below we give a sufficient condition  for \eqref{e-gy1} to be a Hurwitz series, and in Theorem \ref{t-nec} we show that this condition is necessary. In particular, Theorem \ref{t-GAM} implies that \eqref{e-gy1} is a Hurwitz series  for $j=1$ and for $j=k$, so it is a common generalization of Gy's theorem and the Almkvist--Meurman theorem.  However, our proof of Theorem \ref{t-GAM} is more complicated than the proof of the Almkvist--Meurman theorem given in Section \ref{s-AMproof}, and requires the von Staudt--Clausen theorem.

\begin{theorem}
\label{t-GAM}
Let $j$ be a positive integer, and let $h$ and $k$ integers. Suppose that every prime divisor of $j$ also divides $h$ or $k$. Then 
\begin{equation}
\label{e-gyx}
\frac{(e^{hx} -1)^j}{j!} \frac{kx}{e^{kx}-1}
\end{equation} 
is a Hurwitz series.
\end{theorem}

\begin{proof}
Our proof is based on Gy's proof for the case $h=1,j=k$. 

Let $p$ be a prime. A rational number is called \emph{$p$-integral} if its denominator is not divisible by $p$. We call an exponential generating function a \emph{$p$-Hurwitz series} if its coefficients are $p$-integral. Every Hurwitz series is $p$-Hurwitz, and $p$-Hurwitz series are closed under multiplication. It is clear that a power series is a Hurwitz series if and only if it is a $p$-Hurwitz series for every prime $p$.

Let $B(x)=x/(e^x-1)$ be the Bernoulli number generating function. 
We will use two consequences of the von Staudt--Clausen theorem: (i) the denominator of $B_n$ is square-free, and (ii) for any prime $p$, $B_n$ is $p$-integral unless $p-1\mid n$. The first consequence implies that for any integer $k$, $pB(kx)$ is $p$-Hurwitz and that if $p$ divides $k$ then $B(kx)$ is $p$-Hurwitz. 
The second, together with Fermat's theorem, implies that if neither $h$ nor $k$ is divisible by $p$ then  $(k^n-h^n)B_n$ is $p$-integral for all $n$, and thus $B(kx)-B(hx)$ is $p$-Hurwitz. 

To prove the theorem it is sufficient to show that for every prime $p$,  \eqref{e-gyx} is $p$-Hurwitz. We consider three cases (with some overlap): (a) $p\mid h$, (b) $p\mid k$, and  (c) $p\nmid k$ and $p\nmid h$.

(a) Suppose that $p\mid h$. Then since $(e^x-1)^j/j!$ is a Hurwitz series with no constant term, every coefficient of $(e^{hx}-1)^j/j!$ is divisible by $p$.
Thus we may write \eqref{e-gyx} as 
\begin{equation*}
\frac{(e^{hx}-1)^j}{j!\, p}\cdot pB(kx)
\end{equation*}
and both factors are $p$-Hurwitz.

(b) Suppose that $p\mid k$.
Then $B(kx)$ is $p$-Hurwitz, so  \eqref{e-gyx} is $p$-Hurwitz.

(c) Suppose that $p\nmid h$ and $p\nmid k$. Then $p\nmid j$. 
We write \eqref{e-gyx} as
\begin{equation}
\label{e-gy2}
\frac{(e^{hx} -1)^j}{j!} \bigl(B(kx) - B(hx)\bigr) +\frac{(e^{hx} -1)^j}{j!}B(hx).
\end{equation} 
Since $B(kx) - B(hx)$ is $p$-Hurwitz,  so is the first term in \eqref{e-gy2}. The second term in \eqref{e-gy2} may be written
\begin{equation}
\label{e-gy3}
\frac{hx}{j} \frac{(e^{hx}-1)^{j-1}}{(j-1)!}.
\end{equation}
Since $p\nmid j$, \eqref{e-gy3} is also $p$-Hurwitz.
\end{proof}

The condition in Theorem \ref{t-GAM} that every prime divisor of $j$ also divides $h$ or $k$ is necessary, as shown by the following result. 

\begin{theorem}
\label{t-nec}
Let $j$ be a positive integer, and let $h$ and $k$ be integers. Let $p$ be a prime divisor of $j$ that divides neither $h$ nor $k$. Then the coefficient of $x^{j+p-1}/(j+p-1)!$ in \eqref{e-gyx} is not $p$-integral.
\end{theorem}

\begin{proof}
From \eqref{e-gyc}, the coefficient of $x^{j+p-1}/(j+p-1)!$ is
\begin{equation}
\label{e-conv}
\sum_{i=0}^{p-1} \binom{j+p-1}{i}h^{j+p-1-i} S(j+p-1-i,j) k^{i}B_{i}.
\end{equation}
By the von Staudt--Clausen theorem, $B_i$ is $p$-integral for $i<p-1$, but not for $i=p-1$. 
Thus the terms of \eqref{e-conv} with $i<p-1$ are all $p$-integral, so it is sufficient to show that the $i=p-1$ term is not $p$-integral. Since $S(j,j) =1$, the $i=p-1$ term  is 
$\binom{j+p-1}{j}h^j k^{p-1}B_{p-1}$. Neither $h$ nor $k$ is divisible by $p$, so it is enough to show that 
$\binom{j+p-1}{j}$ is not divisible by $p$. To do this we apply Lucas's congruence for binomial coefficients: since $p$ divides $j$ we may set
$j=pr$, and then $\binom{j+p-1}{j} = \binom{pr + p-1}{pr}\equiv \binom{r}{r}\binom{p-1}{0}\equiv 1 \pmod p$.
\end{proof}

\section{Fox's theorem on Euler polynomials} 
There is an analogue of the Almkvist--Meurman theorem for Euler polynomials due to Fox~\cite{fox}.
The Euler polynomials $E_n(u)$ are defined by 
\begin{equation}
\label{e-Eul}
\sum_{n=0}^\infty E_n(u) \frac{x^n}{n!}=\frac{2e^{ux}}{e^x+1}.
\end{equation}
Fox showed that if $r$ and $s$ are integers then 
$s^n\bigl(E_n(r/s) - (-1)^{rs}E_n(0)\bigr)$ is an integer.
Another proof of Fox's theorem was given by Sury \cite{sury2}. We give here a simple  proof of a slight strengthening of Fox's theorem using generating functions.
\begin{theorem} Let $r$ and $s$ be integers, with $s\ne 0$. 
If $s$ is even then $s^nE_{n}(r/s)$ is an integer, and if $s$ is odd then $\tfrac{1}{2} s^n\bigl(E_n(r/s) - (-1)^r E_n(0)\bigr)$ is an integer. 
\end{theorem}
\begin{proof}
We first note that if $f(x)$ is a Hurwitz series with constant term 1 then $1/f(x)$ is also a Hurwitz series.

By Equation \eqref{e-Eul} we have the generating function
\begin{equation}
\label{e-Eul1}
\sumz n s^nE_{n}(r/s)\frac{x^n}{n!}=\frac{2e^{rx}}{e^{sx}+1}.
\end{equation}
If $s$ is even then $2/(e^{sx}+1)$ has integer coefficients since its reciprocal is 
\begin{equation*}
\tfrac12(e^{sx}+1)=1+\sum_{n=1}^\infty \tfrac{1}{2} s^{n}\frac{x^n}{n!},
\end{equation*}
which has integer coefficients and constant term 1. Thus the series in Equation \eqref{e-Eul1} is a Hurwitz series, and in particular, $s^nE_n(0)$ is an integer.

Now suppose that $s$ is odd. We have
\begin{align*}
\sumz n \tfrac12 s^n\bigl(E_n(r/s) - (-1)^r E_n(0)\bigr)\frac{x^n}{n!}
  &=\frac{e^{rx}-(-1)^r}{e^{sx}+1}\\
  &=\frac{e^{rx}-(-1)^r}{e^x+1}\cdot\frac{e^x+1}{e^{sx}+1}\\
  &=\frac{e^{(r-1)x}-e^{(r-2)x}+\cdots + (-1)^{r-1}}{e^{(s-1)x}-e^{(s-2)x}+\cdots+1}. 
\end{align*}
The denominator has constant term 1, so the quotient is a Hurwitz series. 
\end{proof}


\begin{thebibliography}{99}
%\footnotesize

\bibitem{am}
G. Almkvist and A. Meurman, Values of Bernoulli polynomials and Hurwitz's zeta function 
at rational points, {\it C. R. Math. Rep. Acad. Sci. Canada} {\bf 13} (1991), no. 2--3, 104--108.

\bibitem{br}
K. Bartz and J. Rutkowski,
On the von Staudt--Clausen theorem,
{\it C. R. Math. Rep. Acad. Sci. Canada} {\bf 15} (1993), no. 1, 46--48.

\bibitem{carlitz}
L. Carlitz, The Staudt-Clausen theorem, {\it Math. Mag.} {\bf 34} (1960/61), 131--146.

\bibitem{cs}
F. Clarke and I. Sh. Slavutskii,
The integrality of the values of Bernoulli polynomials and of generalised Bernoulli numbers,
{\it Bull. London Math. Soc.} {\bf 29} (1997), no. 1, 22--24.

\bibitem{fox} 
G. J. Fox, Euler polynomials at rational numbers,
{\it C. R. Math. Acad. Sci. Soc. R. Can.} {\bf 21} (1999), no. 3, 87--90. 

\bibitem{gy20}
R. Gy, Bernoulli--Stirling numbers, 
{\it Integers} {\bf 20} (2020), \#A67.

\bibitem{norlund}
N. E. N\"orlund, {\it Vorlesungen \"Uber Differenzenrechnung}, Julius Springer, Berlin, 1924.

\bibitem{oeis}
OEIS Foundation Inc., {\it The On-Line Encyclopedia of Integer Sequences}, Published electronically at \url{https://oeis.org},  2022.
 
\bibitem{sury}
B. Sury, The value of Bernoulli polynomials at rational numbers, {\it Bull. London Math. Soc.} {\bf 25} (1993), 327--329.

\bibitem{sury2}
B. Sury, Values of Euler polynomials, 
{\it C. R. Math. Acad. Sci. Soc. R. Can.} {\bf 23} (2001), no. 1, 12--15.

\bibitem{vandiver1937}
H. S. Vandiver, 
On generalizations of the numbers of Bernoulli and Euler, 
{\it Proc. Natl. Acad. Sci. U.S.A.} {\bf 23} (1937), 555--559.


\bibitem{vandiver1941}
H. S. Vandiver, {\it Simple explicit expressions for generalized Bernoulli numbers of the first order}, Duke Math. J. {\bf 8} (1941), 575--584.

\end{thebibliography}
\end{document}